\documentclass{amsart}
\usepackage[latin1]{inputenc}
\usepackage[T1]{fontenc}

\usepackage[english]{babel}
\usepackage{csquotes}

\usepackage{color} 
\usepackage[colorlinks=true,linktoc=all,linkcolor=blue,citecolor=red,filecolor=pink,urlcolor=blue]{hyperref}

\usepackage{wasysym}

\usepackage{enumitem}
\usepackage[normalem]{ulem}
\usepackage{graphicx}
\graphicspath{{pictures/}}
\usepackage{caption}
\usepackage{subcaption}
\usepackage{transparent}

\usepackage[mathcal,mathscr]{euscript}

\usepackage{tikz}
\usetikzlibrary{arrows}

\usepackage{todonotes}

\usepackage{amssymb,amsmath,latexsym}
\theoremstyle{plain}                    
\newtheorem{theorem}{Theorem}[section]

\newtheorem{lemma}[theorem]{Lemma}
\newtheorem{proposition}[theorem]{Proposition}
\newtheorem{corollary}[theorem]{Corollary}

\theoremstyle{definition}

\newtheorem{choice}[theorem]{Choice of constants}
\newtheorem{assumption}[theorem]{Assumption}
\newtheorem{definition}[theorem]{Definition}

\newtheorem{remark}[theorem]{Remark}
\newtheorem{question}[theorem]{Question}

\numberwithin{equation}{section}

\newcommand{\zz}{\mathbb Z}

\newcommand{\rr}{\mathbb R}

\newcommand{\hh}{\mathbb H}
\newcommand{\ff}{\mathbb F}
\newcommand{\ee}{\mathbb E}
\newcommand{\sphere}{\mathbb S}

\newcommand{\piorb}[1]{\pi_1^{\operatorname{orb}}\left(#1\right)}
\newcommand{\sys}[1]{\operatorname{sys}\left(#1\right)}
\newcommand{\cox}[1]{\operatorname{w}(#1)} 
\newcommand{\wlength}[1]{|#1|_T} 

\newcommand{\lk}[1]{\operatorname{lk}\left(#1\right)}

\newcommand{\rightQ}[2]{\left.\raisebox{.2em}{$#1$}\middle/\raisebox{-.2em}{$#2$}\right.}
\newcommand{\leftQ}[2]{\left.\raisebox{-.2em}{$#2$}\middle\backslash\raisebox{.2em}{$#1$}\right.}

\newcommand{\from}{\colon\thinspace}
\newcommand{\CAT}[1]{\textrm{CAT}(#1)}
\newcommand{\llangle}{\langle\negthinspace\langle}
\newcommand{\rrangle}{\rangle\negthinspace\rangle}
\newcommand{\bw}[2]{\operatorname{BW}_{#1}(#2)}
\newcommand{\injrad}{\operatorname{injrad}}
\newcommand{\acts}{\curvearrowright}

\newcommand{\Stab}[1]{\operatorname{Stab}(#1)}

\newcommand{\coneoff}[1]{\widehat{#1}}
\makeatletter
\@namedef{subjclassname@2020}{\textup{2020} Mathematics Subject Classification}
\makeatother

\begin{document}

\title{Incubulable hyperbolic $3$-pseudomanifold groups}

\author{Jason Manning}
\address{Department of Mathematics, 310 Malott Hall, Cornell University, Ithaca, NY 14853}
\email{jfmanning@cornell.edu}

\author{Lorenzo Ruffoni}
\address{Department of Mathematics and Statistics, Binghamton University, Binghamton, NY 13902, USA}
\email{lorenzo.ruffoni2@gmail.com}

\date{February 2026}

\begin{abstract}
We construct compact hyperbolic $3$-manifolds with totally geodesic boundary, such that the closed $3$-pseudomanifolds obtained by coning off the boundary components are negatively curved and contain locally convex subspaces whose fundamental groups have property (T).
In particular, the fundamental groups of these $3$-pseudomanifolds are word hyperbolic but not cubulable.  We deduce that in any relative cubulation of one of these hyperbolic $3$-manifold groups some hyperplane stabilizer has infinite intersection with the fundamental group of some boundary component.
\end{abstract}

\maketitle

\section{Introduction}
One of the most important open problems in geometric group theory is the Cannon Conjecture, which asserts that any word hyperbolic group with $2$-sphere boundary is virtually cocompact Kleinian.  By work of Bergeron--Wise in one direction, and Markovi\'c and Ha\"issinsky in the other, this conjecture is equivalent to a positive answer to the following.
\begin{question}\label{q:cannon}
  Is every word hyperbolic group with $2$-sphere boundary cubulable?
\end{question}
To say a group is \emph{cubulable} is to say it admits a proper and cocompact action by isometries on some $\CAT{0}$ cube complex.  
Every cocompact Kleinian group is cubulable by \cite[Theorem 1.5]{BW12}.  And a cubulable word hyperbolic group with $2$-sphere boundary is virtually cocompact Kleinian by~\cite[Theorem 1.10]{Ha15} (generalizing~\cite{Mar13}).

We do not directly speak to Question~\ref{q:cannon}, but we study the corresponding question for groups with \emph{Pontryagin sphere} boundary.  The Pontryagin sphere is a certain inverse limit of surfaces which appears as the boundary at infinity of many $\CAT{0}$ $3$-dimensional complexes (universal covers of $3$-\emph{pseudomanifolds}, see below). 
It is like a sphere in that it is a connected $2$-dimensional compactum whose $2$-dimensional \v{C}ech cohomology is cyclic.  However its $1$-dimensional cohomology is infinitely generated. We refer the reader to \cite{JA91,DR99,FI03,SW20}.
\begin{question}\label{q:pcannon}
  Is every word hyperbolic group with Pontryagin sphere boundary cubulable?
\end{question}
Classical examples of word hyperbolic groups with Pontryagin sphere boundary occur for instance as finite index subgroups of a right-angled Coxeter group (RACG) whose nerve is a flag-no-square triangulation of a connected closed orientable surface, and via Charney--Davis strict hyperbolization \cite{DA08,CD95}. 
Both sources of examples yield groups that are cubulable (see \cite{DA08,LR24}). 
Nonetheless, we answer Question~\ref{q:pcannon} in the negative; see \eqref{item:notcubulable} in Corollary~\ref{cor:main}.

Our examples, as well as the aforementioned classical examples, are fundamental groups of negatively curved pseudomanifolds.
A \textit{$3$-pseudomanifold} is a $3$-dimensional polyhedral complex $P$ in which the links of vertices are connected closed orientable surfaces of arbitrary genus.
These arise by taking a compact $3$-manifold $M$ with orientable boundary and coning off the boundary components (one cone point for each component) to obtain a compact space $\coneoff M$; see \S\ref{sec:coneoff} for details.

A \emph{negatively curved} $3$-pseudomanifold is one which admits a complete locally $\CAT{-1}$ metric. 
Such a pseudomanifold is aspherical.  If it is also compact, then the fundamental group is word hyperbolic.  Our definition does not permit pseudomanifolds with boundary, so for us a pseudomanifold will be \emph{closed} if and only if it is compact.  A closed pseudomanifold has a fundamental class over $\zz/2\zz$, and so the fundamental group of a closed aspherical $3$-pseudomanifold is $3$-dimensional.

Our main result is the following.

\begin{theorem}\label{thm:main}
    There exists a sequence of compact hyperbolic $3$-manifolds $M_n$ with totally geodesic boundary such that the boundary cone-offs $\coneoff M_n$ satisfy the following:
\begin{enumerate}
    \item \label{item:negative} $\coneoff M_n$ is a closed negatively curved $3$-pseudomanifold.
    \item \label{item:T} $\pi_1(\coneoff M_n)$ contains an infinite quasiconvex subgroup with property (T).
    \item \label{item:char} The genus of the components of $\partial M_n$ and the Euler characteristic of $\coneoff M_n$ both go to infinity with $n$.
\end{enumerate}
\end{theorem}

The existence of an infinite subgroup with property (T) is a well-known obstruction to the existence of nice geometric actions, and makes these groups quite different from $3$-manifold groups.
More precisely, we have the following.
(See \cite{BHV08} for the definitions of the Haagerup and (T) properties, and \cite{AG08} for the definition of RFRS.)

\begin{corollary}\label{cor:main}
    In the same setting as Theorem~\ref{thm:main}, we have the following.
\begin{enumerate}
    \item  \label{item:notcubulable} $\pi_1(\coneoff M_n)$  does not act properly by cubical isometries on a  $\CAT 0$ cube complex. In particular, it is not cubulable.

    \item \label{item:notHaagerup} $\pi_1(\coneoff M_n)$ does not have the Haagerup property. In particular, it does not act properly by isometries on a real or complex hyperbolic space.

    \item \label{item:notRFRS} $\pi_1(\coneoff M_n)$ is not virtually RFRS.
\end{enumerate}
\end{corollary}

\begin{remark}[Virtual cubulability]
The non-cubulability of $\pi_1(\coneoff M_n)$ does not automatically imply the non-cubulability of fundamental groups of $3$-pseudomanifolds obtained by coning off the boundary components of an arbitrary finite-sheeted cover of $M_n$.  Indeed, one may use Wise's Malnormal Special Quotient Theorem \cite[Theorem 12.2]{Wise21} together with the cubulability of $\pi_1(M_n)$ \cite[Theorem 17.14]{Wise21} to see that each $M_n$ has some finite sheeted cover $M_n'$ so that $\pi_1(\coneoff M_n')$ \emph{is} cubulable. 
\end{remark}

\begin{remark}[Gromov boundary]\label{rmk:boundary}
    The Gromov boundary of the groups $\pi_1(\coneoff M_n)$ in Theorem~\ref{thm:main} is the tree of manifolds defined by closed orientable surfaces of positive genus, i.e., 
    a Pontryagin sphere, see \cite[Theorem A.1]{KM} and \cite{SW20}.
    Moreover, the Gromov boundary of the infinite quasiconvex subgroup with property (T) has limit set a Menger curve, see \cite{KK00}.
\end{remark}

\begin{remark}[Property (T) vs Haagerup Property]\label{rmk:T}
The groups $\pi_1(\coneoff M_n)$ in Theorem~\ref{thm:main} 
split over quasiconvex surface subgroups (see Proposition~\ref{prop:split}), so they do not have property (T).
In particular, they are $3$-dimensional word hyperbolic groups without (T) and without Haagerup, whose boundary is connected and has cyclic top-dimensional \v Cech cohomology (since the boundary is a Pontryagin sphere, see \cite[Theorem 6.2]{SW20}).
See \cite[Remark 5.6]{LR25} for other examples in dimension  $\geq 9$, whose boundary is a topological sphere.
For examples of non-Kleinian hyperbolic 3-pseudomanifold groups that have the Haagerup property, and even act convex cocompactly on $\hh^n$ for some $n>3$, see the RACGs considered in  \cite{DLMR25}. 
\end{remark}

\begin{remark}[Relatively geometric cubulation]\label{rem:agol}
    Ian Agol pointed out the following consequence of our construction (see Proposition~\ref{prop:agol}):
    For any relatively geometric cubulation (in the sense of \cite{EG20}) of any of the Kleinian groups $\pi_1(M_n)$ that we construct, there must be a hyperplane stabilizer with infinite intersection with some boundary subgroup of $\pi_1(M_n)$.
    This is in contrast to the situation for finite volume cusped $3$--manifolds.  The fundamental groups of such manifolds admit relatively geometric cubulations by quasi-Fuchsian closed surface subgroups  by the ubiquity results in~\cite{CF19,KW21}.
\end{remark}

\subsection*{Outline of the paper}
Our proof strategy can be summarized as follows.  We first describe a sequence $T_n$ of simplicial $2$-complexes whose fundamental groups have property (T), and so that as $n$ tends to infinity, the girth of links of vertices in $T_n$ also tends to infinity.  For large $n$ we will embed these complexes $\pi_1$-injectively into negatively curved $3$-pseudomanifolds.  To embed such a complex, we first delete a regular neighborhood of the vertices, and thicken the resulting hexagon complex to a $3$-dimensional handlebody.  Mirroring the boundary of this handlebody appropriately, we obtain a $3$-orbifold with boundary $H_n$ (usually referred to in the sequel as $H$).  The boundary of $H_n$ contains a $\pi_1$-injective graph corresponding to the union of links of vertices of $T_n$.  The $3$-orbifold $H_n$ is (orbifold) covered by a hyperbolic $3$-manifold $M$ with totally geodesic boundary, which we cone off to obtain our $3$-pseudomanifold.  The handlebody $H_n$ lifts to $M$, giving rise to a subset of the cone-off homotopy equivalent to $T_n$.

In \S\ref{sec:prelim} we fix our notation and terminology about the construction of pseudomanifolds as cone-offs of manifolds and about orbifolds. Moreover, we collect some preliminary material from \cite{LMW19} about the $2$-dimensional simplicial complexes $T_n$ whose fundamental groups have property (T).
In \S\ref{sec:main orbifold} we present the construction of a particular hyperbolic $3$-orbifold with boundary, and in \S\ref{sec:covers} we show how to construct suitable finite covers  for which we can obtain quantitative control on various geometric quantities. 
The orbifold structure constructed in \S\ref{sec:main orbifold} is given by orthogonal mirrors. This is reminiscent of the Davis reflection trick \cite{DA83}. However, we use the mirror structure to construct manifolds with boundary instead of closed manifolds.
The geometric control enables us to use the results from \cite{KM} to construct negatively curved metrics on the cone-offs of such covers. The proofs of the main theorem and corollary stated in the Introduction are presented in \S\ref{sec:proofs}. We end the paper with some open questions in \S\ref{sec:questions}.

\vspace{.25cm}

\noindent \textbf{Acknowledgments.}
J.M. was partially supported by the Simons Foundation, grant \#942496.
L.R. was partially supported by INDAM-GNSAGA.
We thank Chris Hruska for pointing us to useful references.  We thank Ian Agol for pointing out Proposition~\ref{prop:agol}, and Daniel Groves for useful conversations.

\section{Preliminaries}\label{sec:prelim}

\subsection{Coning off}\label{sec:coneoff}
Here we state the results from~\cite{KM} which allow us to put negatively curved (i.e., locally $\CAT{k}$ for some $k<0$) metrics on our $3$-pseu\-do\-man\-i\-folds. 
Let $M$ be a compact $3$-manifold with boundary. 
The \textit{cone-off} is the   space
\[ \coneoff{M} = M \sqcup (\partial M\times [0,1])/\sim \]
where $\sim$ is the equivalence relation generated by
\begin{itemize}
\item $x\sim (x, 1)$ if $x\in \partial M$; and
\item $(x, 0)\sim (y,0)$ if $x$ and $y$ lie in the same component of $\partial M$.
\end{itemize}
The resulting space $\coneoff M$ is a closed $3$-pseudomanifold (since $3$-manifolds can be triangulated).
If $Z\subset M$, then we define the \textit{induced cone-off} $\coneoff{Z}$ to be the subset which is the image of
\[ Z\sqcup (\partial M\cap Z) \times [0,1] \]
in the quotient space $\coneoff{M}$.

Sufficient geometric conditions are given in~\cite{KM} for $\coneoff{M}$ to admit a negatively curved metric $\hat{d}$ and 
 for the induced cone-off of a subspace $Z$ to be locally convex.
To state the conditions we need a definition from~\cite{KM}: 
\begin{definition}
  Let $X$ be a geodesic space, and let $\Upsilon\subset X$ be closed.  The \emph{buffer width} of $\Upsilon$ in $X$, written $\bw{X}{\Upsilon}$, is half the length of the shortest nondegenerate local geodesic in $X$ intersecting $\Upsilon$ only in its endpoints.  If there is no such local geodesic, then $\bw{X}{\Upsilon} = \infty$. 
\end{definition}
Here are a couple of special cases. If $X$ is a closed non-positively curved Riemannian manifold and $x\in X$, then $\bw{X}{\{x\}}$ is the injectivity radius at $x$. More generally, if $Z$ is an embedded totally geodesic submanifold of  $X$, then $\bw{X}{Z}$ is the normal injectivity radius of $Z$.

The first result is about when the boundary cone-off $\coneoff{M}$ of a compact hyperbolic $3$-manifold $M$ can be given a negatively curved metric.
For $A\subset M$, we denote by $N_b(A)$ the open neighborhood of radius $b$ around $A$.
\begin{theorem} \label{thm:KM1} \cite[Theorem A]{KM}
 Let $M$ be a compact hyperbolic manifold with totally geodesic boundary, let $b$ be a positive number less than $\bw{M}{\partial M}$ and let $c>\pi/\sinh(b)$.
 Suppose \[\injrad(\partial M) > c.\]
  Then there is a negatively curved metric $\hat{d}$ on $\coneoff{M}$ and an embedding of $M\smallsetminus N_b(\partial M)$ into $(\widehat{M},\hat{d})$ which is a local isometry with image equal $M \subset \coneoff M$.
\end{theorem}

For $b$ as in the theorem, the normal exponential map gives a diffeomorphism $\eta\from \partial M \times [0,b) \to N_b(\partial M)$; the restriction $\eta|\{x\}\times [0,b)$ gives a geodesic orthogonal to $\partial M$ at $x$.  
\begin{definition}\label{def:tame}
    The set $Z\subset M$ \emph{is a tame product near $N_b(\partial M)$} if, for some $b'>b$, $Z\cap N_{b'}(\partial M) = \eta( (Z\cap \partial M)\times [0,b') )$.
\end{definition}  The following is a special case of a result proved in~\cite{KM}:
\begin{theorem} \label{thm:KM2} \cite[Theorem B]{KM}
  Let $M$, $b$, and $c$ satisfy the hypotheses of Theorem~\ref{thm:KM1} and let $\hat{d}$ be the negatively curved metric on $\coneoff{M}$ constructed in that theorem.  Suppose $Z\subset M$ is closed, locally convex, is a tame product near $N_b(\partial M)$, and that
  \[ \bw{\partial M}{Z\cap \partial M} > \frac{c}{2}. \]
  Then $\coneoff{Z}$ is isotopic to a locally convex set in $(\coneoff{M},\hat{d})$.
\end{theorem}

\subsection{Orbifolds}\label{sec:orbifolds}
For an introduction to the general notion of a smooth orbifold see~\cite{ BB12,BH99,BMP,CHK,Kap01}.  We briefly recall the idea:
An $n$-orbifold (possibly with boundary) consists of a topological space (the \emph{underlying} space) together with an open cover by sets of the form $U = \tilde{U}/\Gamma$, where $\tilde{U}$ is an open subset of $\rr^n$ or $\rr^{n-1}\times[0,\infty)$, and $\Gamma$ is a finite group of diffeomorphisms with nonempty fixed point set.  (Smoothness also implies that a nonempty intersection between these sets gives rise to smooth transition maps.)
If a point $x$ in the underlying space is the image of a fixed point of $\Gamma$, we say that the \emph{isotropy group} of $x$ is $\Gamma$.  (Strictly speaking we  replace $\Gamma$ by the finite subgroup of $O(n)$ given by taking the derivative of its elements; this gives a well-defined group up to conjugacy in $O(n)$.) The \emph{boundary} $\partial X$ is the $(n-1)$-orbifold whose underlying space consists of those points where the orbifold chart always comes from $\rr^{n-1}\times[0,\infty)$, and whose isotropy groups are the isotropy groups for $X$, restricted to $\rr^{n-1}\times\{0\}$.  The isotropy groups of the points in the boundary lie in $O(n-1)$.
The orbifold $X$ is said to be \emph{closed} if its underlying space is compact and $\partial X$ is empty.

In this paper we are interested only in smooth orbifolds with isotropy groups which are trivial or finite reflection groups.  These are sometimes called \emph{locally reflective} orbifolds.
The underlying space of such an orbifold will always be a manifold with boundary.  We refer to the boundary of the underlying space as the \emph{topological boundary}.  The underlying space of the boundary orbifold
$\partial X$
 is the closure of the set of topological boundary points with trivial isotropy group.

A \emph{mirror} in a reflective orbifold is the closure of a maximal connected subset in which the isotropy group at each point is $\zz/2$.

A $2$-dimensional locally reflective orbifold $\Sigma$ has underlying space equal to a surface with boundary.  This (topological) boundary decomposes as a union of mirrors and components of $\partial \Sigma$.  If $\Sigma$ is closed (as it always will be for us) the topological boundary is just a union of mirrors.
Where two mirrors intersect, the isotropy group is dihedral of order $2 n$ for some $n \ge 2$.  

A $3$-dimensional locally reflective orbifold $X$ has underlying space a $3$-manifold with boundary.  The topological boundary decomposes as a union of mirrors and components of $\partial X$.  Each component of $\partial X$ is a locally reflective $2$-orbifold.  The mirrors of $X$ are divided from each other and from $\partial X$ by some trivalent graph on the topological boundary.  Each edge of this graph consists of points with dihedral or $\zz/2\zz$ isotropy group, and each vertex has isotropy group which is some $3$-dimensional reflection group.  This reflection group is dihedral if the vertex is adjacent to a component of $\partial X$; otherwise it is a spherical triangle reflection group.

Suppose that $\Gamma$ is a discrete group of isometries of some \emph{geometry} $\tilde X$ (for example $\ee^n$, $\sphere^n$ or $\hh^n$).
The quotient $X = \leftQ{\tilde X}{\Gamma}$ naturally receives the structure of an orbifold. Such an orbifold is said to be \emph{geometric}
 (or \emph{Euclidean}, \emph{spherical}, \emph{hyperbolic}, etc. if we wish to be more specific).  The quotient map $\widetilde X\to X$ is an example of an \emph{orbifold covering map}.
 
  For example, the quotient of a 2-dimensional geometry by an $(l,m,n)$ triangle reflection group is an orbifold with underlying space a disc, with three points on the topological boundary having dihedral isotropy groups of orders $2l,2m,2n$, and all other boundary points having isotropy groups of order two.  
      
When $X$ is the quotient of a geometry by $\Gamma$, the \emph{orbifold fundamental group} $\piorb X$ may be identified with $\Gamma$, but may also be defined in terms of homotopy classes of orbifold loops at a basepoint.
For more on orbifold covering maps and the relation with orbifold fundamental group see \cite{CHK}, \cite[Chapter III.$\mathcal G$]{BH99}, \cite[Chapter 13]{RA19}.

\subsection{The seed with property (T)}\label{sec:turnover groups}
We describe here a construction which is closely related to the examples in \cite{LMW19}.  In particular we will describe some fundamental groups of triangle complexes which are subgroups of finite index in some of the groups considered there.

\begin{choice}\label{choice:prime}
We fix $k\geq 18$ such that $k-1$ is prime and congruent to $1$ mod $4$.
    (For concreteness the reader may suppose that $k = 18$. Our pictures will pretend $k=4$ for simplicity.)
\end{choice}

  Let $\Theta$ be the $1$-complex with two vertices and $k$ edges between them.  The following is a consequence of work of Lubotzky--Phillips--Sarnak~\cite{LPS88} and Margulis~\cite{MA88}.
\begin{theorem}\label{thm:tower}
  There is a tower of finite regular covers $\cdots \Lambda_n \to \Lambda_{n-1} \to \cdots \to \Theta$ with the following properties:
  \begin{enumerate}
    \item $\lim_{n\to\infty}\operatorname{girth}(\Lambda_n)= \infty$.
    \item For each $n$, the first eigenvalue of the (normalized) Laplacian on $\Lambda_n$ is larger than $\frac{1}{2}$.
\end{enumerate}
\end{theorem}
\begin{proof}
  This is explained in~\cite{LMW19}, following closely \cite[Chapter 7]{Lub94}.
  For the fact that the graphs can be taken in a tower, see the bottom of page 64 of~\cite{LMW19}.  The reason is that the subgroups of $\pi_1(\Theta)$ corresponding to the covers $\Lambda_n\to \Theta$ arise as congruence subgroups $\Gamma(2q_n)$ for an arithmetic group $\Gamma\cong \pi_1(\Theta)$, and the numbers $q_n$ can be taken to be successive powers of a prime $q$ where $k-1$ is not a square modulo $q$.  (For example if $k-1 = 17$ we could take $q_n = 3^n$ or $q_n = 5^n$.)
\end{proof}

Let $\Delta$ be a $2$-simplex.  A \emph{$k$-fold turnover} is the complex
\[ \Delta\times\{1,\ldots,k\}/ (x,i)\sim (x,j)\mbox{ when }x\in\partial \Delta. \]

Let $T_0$ be the  hexagonal $2$-complex obtained from a $k$-fold turnover by chopping off the three tips.\footnote{In other words $T_0$ is equal to a $k$-fold turnover minus an open regular neighborhood of its $0$-skeleton.}
Any two hexagons in $T_0$ form a pair of pants.
The boundary of $T_0$ in the turnover consists of three copies of $\Theta$. 
We write $P_1,P_2,P_3$ for the fundamental groups of these copies of $\Theta$.  (The reader who is concerned about basepoints here should imagine that all the basepoints lie in the hexagon contained in $\Delta\times\{k\}$.)
Since $T_0$ deformation retracts to a copy of the complete bipartite graph $K_{3,k}$, its fundamental group is free, specifically
 $\pi_1(T_0) = \ff_{2k-2}$.

Fix a basis $x_1,\cdots,x_{k-1}$ of $\pi_1 (\Theta)$ and isomorphisms $\psi_i\from \pi_1 (\Theta) \to P_i$ for $i = 1,2,3$, so that  the relations $\psi_1(x_j)\psi_2(x_j)\psi_3(x_j)$ are satisfied for $j=1,\dots,k-1$.  (These are the relations coming from the pairs of pants in $T_0$.)  

Let $\tau\from \pi_1 (\Theta)\to \pi_1 (\Theta)$ be the automorphism sending each $x_j$ to $x_j^{-1}$.  Note that $\tau$ is induced by the graph automorphism of $\Theta$ swapping the two vertices and sending each edge to itself.
Let $Q_n$ be the quotient of $\pi_1(\Theta)$ by the normal subgroup corresponding to the cover $\Lambda_n\to \Theta$ from Theorem~\ref{thm:tower}, and let $\phi_n\from \pi_1(\Theta)\to Q_n$ be the natural quotient map.
Now define a map $\Phi_n\from \pi_1( T_0)\to Q_n\times Q_n$ by defining it consistently on the (generating) subgroups $P_i$:
\begin{align*}
\Phi_n|_{P_1} & = (\phi_n\circ\psi_1^{-1},1)\\
\Phi_n|_{P_2} & = (1,\phi_n\circ\psi_2^{-1})\\
\Phi_n|_{P_3} & = (\phi_n\circ\tau\circ\psi_3^{-1},\phi_n\circ\tau\circ\psi_3^{-1})
\end{align*}

\begin{definition}\label{def:Tn}
    Let $T_n$ be the finite cover of $T_0$ corresponding to the kernel of $\Phi_n$.
\end{definition}

Using the isomorphism  $\psi_i\from \pi_1 (\Theta)\to P_i$, we can identify $\ker\Phi_n|_{P_i}=\ker \Phi_n \cap P_i$ 
 with $\ker \phi_n$ for $i=1,2$ and with $\tau(\ker \phi_n)=\ker (\phi_n\circ \tau)$ for $i=3$.

\begin{lemma}\label{lem:turnover (T)}
  Let $\coneoff{T}_n$ be the triangle complex obtained by coning off the elevations of $\Theta$ in $T_n$.  If $n$ is sufficiently large, then the link of every vertex of $\coneoff{T}_n$ has girth at least $6$ and $\pi_1(\coneoff{T}_n)$ has property (T).
\end{lemma}
\begin{proof}
 Since $\tau$ is induced by a graph automorphism of $\Theta$, all the vertex links in $\coneoff T_n$ are isomorphic as graphs to the cover $\Lambda_n$.  
By Theorem~\ref{thm:tower}, for $n$ large enough we can assume that the link of every vertex in $\coneoff T_n$ has girth at least $6$.
Now, if we identify each triangle of $\coneoff T_n$ with an equilateral Euclidean triangle, then the induced metric is locally $\CAT 0$, so $\coneoff{T}_n$ is aspherical.
This asphericity, together with the condition on the Laplacian of $\Lambda_n$ in Theorem~\ref{thm:tower}, allows us to apply \cite[Corollary 1]{BS97} (or \cite[Th\'eor\`eme 1]{Zuk96}) to deduce that $\pi_1(\coneoff {T}_n)$ has property (T).
\end{proof}

\section{The handlebody \texorpdfstring{$H_0$}{H0} and its hyperbolic orbifold structure} \label{sec:main orbifold}

Let $H_0$ be the $3$-manifold with boundary  obtained by thickening up $T_0$ in $\rr^3$, see Figure~\ref{fig:thickened}.
Note that $H_0$ is homotopy equivalent to $T_0$, so $\pi_1(H_0) = \ff_{2k-2}$.
(Recall $k\geq 18$ was fixed in the Choice of constants~\ref{choice:prime}.)
More precisely, $H_0$ is a handlebody of genus $g=2k-2$, whose topological boundary is a closed surface of genus $g=2k-2$, decomposed into
\begin{itemize}
    \item a subsurface $S_0$ whose connected components are three $k$-holed spheres, arising from the peripheral $\Theta$-graphs of $T_0$, and
    \item a subsurface $Y_0$ whose connected components are  $k$ pairs of pants, arising from the ``piping'' between $\Theta$-graphs.
\end{itemize}

\begin{figure}[ht]
    \centering
    \def\svgwidth{.9\columnwidth}
    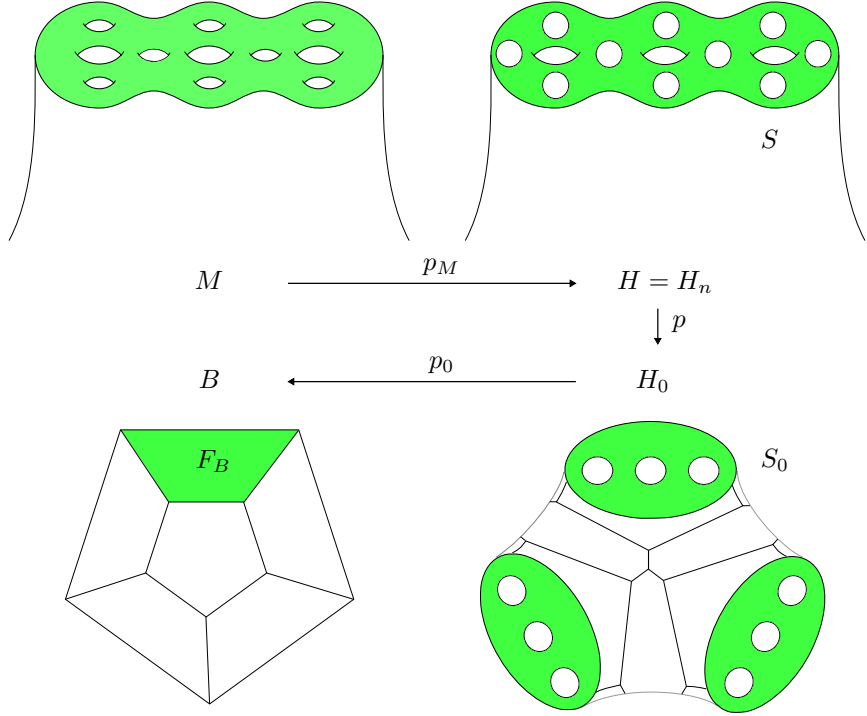
    \caption{The main covering steps in the construction. (The picture only shows a portion of $H$ and $M$.)
    The boundaries of $M$ (as a manifold) and of $H$, $H_0$, and $B$ (as orbifolds) are represented in green. The boundary of $M$ consists of closed surfaces,  the boundaries of $H$ and $H_0$ consist of compact surfaces with boundary, and the boundary of $B$ consists of the face $F_B$.}
    \label{fig:coverings}
\end{figure}

    To recall the strategy of our main construction: 
    We will first take the cover $p:H_n\to H_0$ corresponding to a cover $T_n\to T_0$ from Definition~\ref{def:Tn}, with $n$ large enough to apply Lemma~\ref{lem:turnover (T)}.
    We will then take an orbifold cover $M$ of $H=H_n$ to which $H=H_n$ lifts homeomorphically, and then cone off $\partial M$.  Choosing these covers carefully, and using the results from \cite{KM} (see \S\ref{sec:coneoff}) the inclusion of the induced cone-off $\widehat H$ into $\widehat M$ will have $\pi_1$-image isomorphic to $\pi_1(\widehat T_n)$, which has property (T) by Lemma~\ref{lem:turnover (T)}.
    
    In order to maintain control of the geometry, there will also be a further orbifold cover to a partially mirrored polyhedron $B$, which we use to endow $H_0$, $H=H_n$, and $M$ with nice hyperbolic structures; see Figure~\ref{fig:coverings}. For a more detailed picture of the polyhedron $B$, see  Figure~\ref{fig:prism}.

\subsection{The hyperbolic orbifold structure on \texorpdfstring{$H_0$}{H0}}\label{sec:hyperbolization}
We want to endow $H_0$ with the structure of a hyperbolic $3$-orbifold with boundary (see \S\ref{sec:orbifolds} for definitions). 
To do this, we will endow a certain quotient $B$ of $H_0$ with the structure of a reflective $3$-orbifold with boundary.  We will then use Andreev's Theorem on hyperbolic polyhedra to find a hyperbolic structure on this orbifold.
Finally we will lift this structure to $H_0$.
Each mirror of $H_0$  will be a finite-sided polygon in $Y_0$ and any two mirrors will either be disjoint or meet at right angles;  the surfaces in $S_0$ will be left unmirrored but they will also make right angles with their mirrored neighbors in $Y_0$.
To determine an appropriate orbifold structure, we choose 
a trivalent graph on the topological boundary of $H_0$ so that the non-simply connected regions in the complement are the components of $S_0$.
There are many ways to do this, but we choose a particular construction for definiteness.
 
For each pair of pants in $Y_0$, consider the combinatorial tiling into pentagons and octagons shown in Figure~\ref{fig:patterned_pants}.
The resulting decomposition of the boundary of $H_0$ is pictured in Figure~\ref{fig:thickened}.  The handlebody $H_0$ together with this decoration has symmetry group $G=D_3\times D_k$, where $D_k$ denotes the dihedral group of order $2k$.
The quotient of $H_0$ by $G$ is a pentagonal prism $B$, depicted in Figure~\ref{fig:prism}.
See Figure~\ref{fig:steps} for an illustration of the steps involved in taking this quotient.
The image of $S_0$ is the green face $F_B$ and the image of the topological boundary of $S_0$ is the bold edge $e$ on the left.

\begin{figure}[htbp]
  \centering
  \def\svgwidth{.5\columnwidth}
  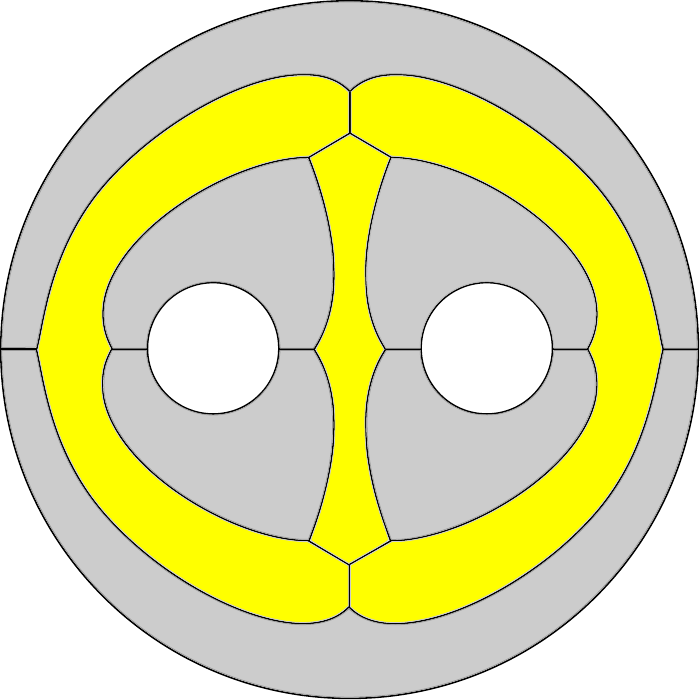
  \caption{Each pair of pants in $Y_0$ in the topological boundary of $H_0$ is cellulated by pentagons (gray) and octagons (yellow) as shown.}  
  \label{fig:patterned_pants}
\end{figure}

\begin{figure}[htbp]
  \centering
\def\svgwidth{.6\columnwidth}
  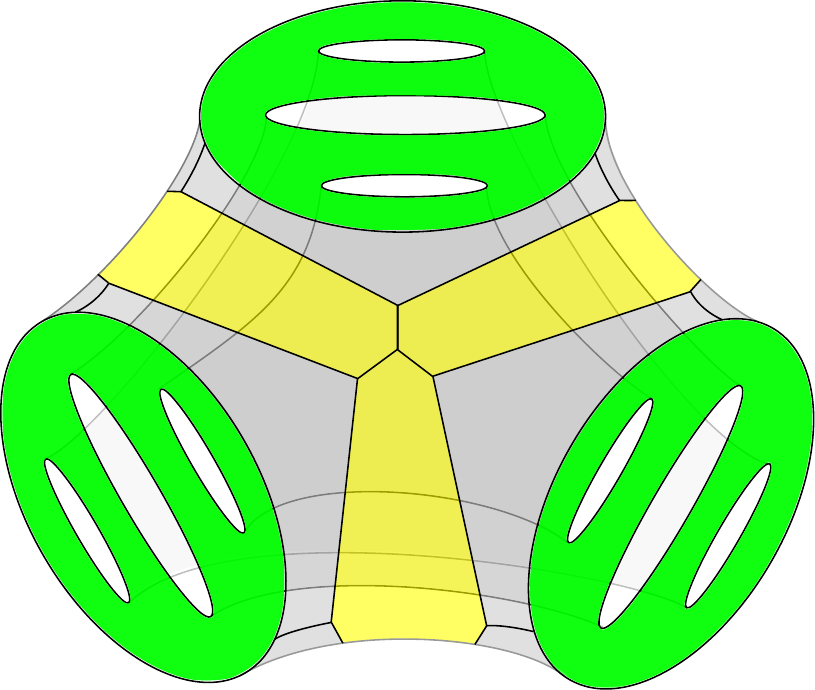
  \caption{A mirror structure on $H_0$, with $k=4$. The mirror pattern of $Y_0$ (see  Figure~\ref{fig:patterned_pants}) is shown only on the front of the outside pair of pants, but it should be replicated on the other $k-1$ pairs of pants.  The boundary of $H_0$ as an orbifold is colored green.}
  \label{fig:thickened}
\end{figure}

\begin{figure}[htbp]
  \centering
  \def\svgwidth{.5\columnwidth}
\begingroup%
  \makeatletter%
  \providecommand\color[2][]{%
    \errmessage{(Inkscape) Color is used for the text in Inkscape, but the package 'color.sty' is not loaded}%
    \renewcommand\color[2][]{}%
  }%
  \providecommand\transparent[1]{%
    \errmessage{(Inkscape) Transparency is used (non-zero) for the text in Inkscape, but the package 'transparent.sty' is not loaded}%
    \renewcommand\transparent[1]{}%
  }%
  \providecommand\rotatebox[2]{#2}%
  \newcommand*\fsize{\dimexpr\f@size pt\relax}%
  \newcommand*\lineheight[1]{\fontsize{\fsize}{#1\fsize}\selectfont}%
  \ifx\svgwidth\undefined%
    \setlength{\unitlength}{215.64927864bp}%
    \ifx\svgscale\undefined%
      \relax%
    \else%
      \setlength{\unitlength}{\unitlength * \real{\svgscale}}%
    \fi%
  \else%
    \setlength{\unitlength}{\svgwidth}%
  \fi%
  \global\let\svgwidth\undefined%
  \global\let\svgscale\undefined%
  \makeatother%
  \begin{picture}(1,0.95629016)%
    \lineheight{1}%
    \setlength\tabcolsep{0pt}%
    \put(0,0){\includegraphics[width=\unitlength,page=1]{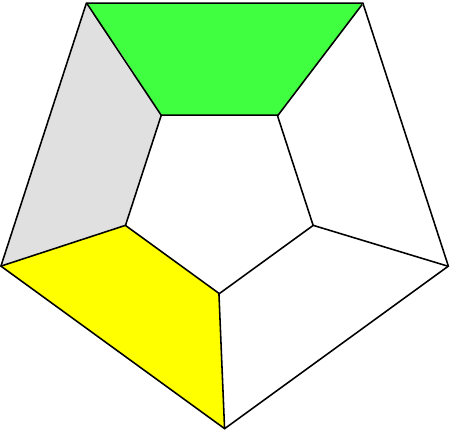}}%
    \put(0.67985497,0.57362441){\makebox(0,0)[lt]{\lineheight{1.25}\smash{\begin{tabular}[t]{l}$k$\end{tabular}}}}%
    \put(0.81733762,0.14453965){\makebox(0,0)[lt]{\lineheight{1.25}\smash{\begin{tabular}[t]{l}$3$\end{tabular}}}}%
    \put(0.23652703,0.56695358){\makebox(0,0)[lt]{\lineheight{1.25}\smash{\begin{tabular}[t]{l}$4$\end{tabular}}}}%
    \put(0.14450295,0.14022584){\makebox(0,0)[lt]{\lineheight{1.25}\smash{\begin{tabular}[t]{l}$4$\end{tabular}}}}%
    \put(0,0){\includegraphics[width=\unitlength,page=2]{prism2.pdf}}%
    \put(0.21565984,0.76867007){\makebox(0,0)[lt]{\lineheight{1.25}\smash{\begin{tabular}[t]{l}$e$\end{tabular}}}}%
    \put(0.45215503,0.81736021){\makebox(0,0)[lt]{\lineheight{1.25}\smash{\begin{tabular}[t]{l}$F_B$\end{tabular}}}}%
    \put(0.26435001,0.26785678){\makebox(0,0)[lt]{\lineheight{1.25}\smash{\begin{tabular}[t]{l}$F_\Sigma$\end{tabular}}}}%
  \end{picture}%
\endgroup%

  \caption{The pentagonal prism $B$.  The unlabeled edges have angle $\pi/2$.  Edges labeled by an integer $n$ have angle $\pi/n$. 
  The green face $F_B$ is the image of $S_0$. Its bold edge $e$ is the image of the topological boundary  of $S_0$. The yellow face $F_\Sigma$ is the image of the octagonal faces in $Y_0$.}
  \label{fig:prism}
\end{figure}

\begin{figure}[htbp]
  \centering
  \def\svgwidth{\columnwidth}
\begingroup%
  \makeatletter%
  \providecommand\color[2][]{%
    \errmessage{(Inkscape) Color is used for the text in Inkscape, but the package 'color.sty' is not loaded}%
    \renewcommand\color[2][]{}%
  }%
  \providecommand\transparent[1]{%
    \errmessage{(Inkscape) Transparency is used (non-zero) for the text in Inkscape, but the package 'transparent.sty' is not loaded}%
    \renewcommand\transparent[1]{}%
  }%
  \providecommand\rotatebox[2]{#2}%
  \newcommand*\fsize{\dimexpr\f@size pt\relax}%
  \newcommand*\lineheight[1]{\fontsize{\fsize}{#1\fsize}\selectfont}%
  \ifx\svgwidth\undefined%
    \setlength{\unitlength}{1257.811039bp}%
    \ifx\svgscale\undefined%
      \relax%
    \else%
      \setlength{\unitlength}{\unitlength * \real{\svgscale}}%
    \fi%
  \else%
    \setlength{\unitlength}{\svgwidth}%
  \fi%
  \global\let\svgwidth\undefined%
  \global\let\svgscale\undefined%
  \makeatother%
  \begin{picture}(1,0.26403647)%
    \lineheight{1}%
    \setlength\tabcolsep{0pt}%
    \put(0,0){\includegraphics[width=\unitlength,page=1]{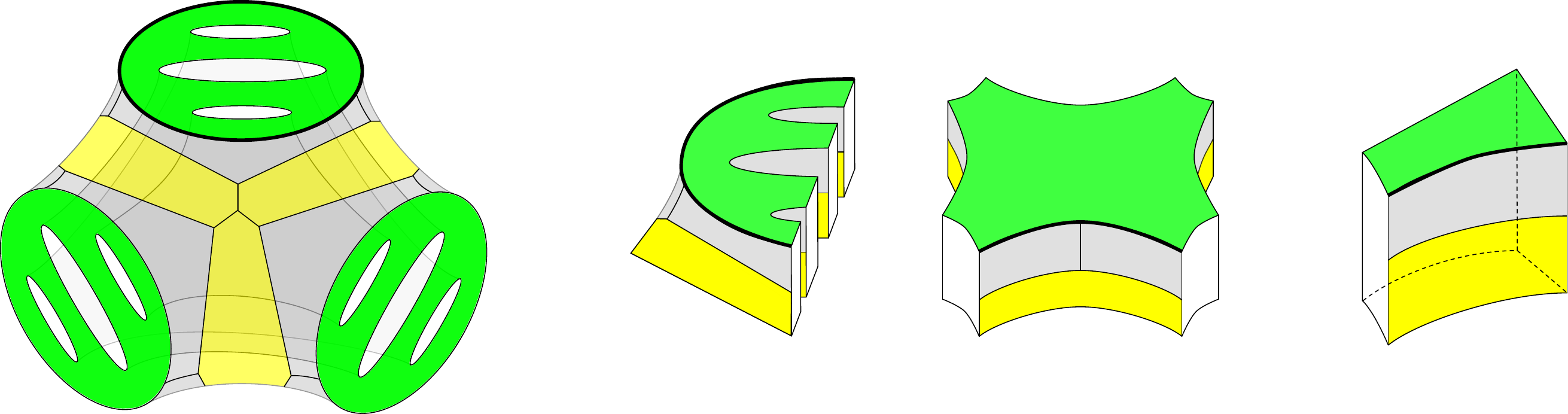}}%
    \put(0.57321634,0.12268039){\color[rgb]{0,0,0}\makebox(0,0)[t]{\lineheight{1.25}\smash{\begin{tabular}[t]{c}=\end{tabular}}}}%
    \put(0.35873,0.12735958){\color[rgb]{0,0,0}\makebox(0,0)[t]{\lineheight{1.25}\smash{\begin{tabular}[t]{c}$\overset{D_3}{\longrightarrow}$\end{tabular}}}}%
    \put(0.82382326,0.12735958){\color[rgb]{0,0,0}\makebox(0,0)[t]{\lineheight{1.25}\smash{\begin{tabular}[t]{c}$\overset{D_k}{\longrightarrow}$\end{tabular}}}}%
  \end{picture}%
\endgroup%

  \caption{From left to right: The handlebody $H_0$, with $k=4$;  the quotient of $H_0$ by the $D_3$-symmetry; a different representation of the same quotient, in which  the $D_k$-symmetry is more clear; the further quotient by the $D_k$-symmetry, i.e., the pentagonal prism $B$, also represented in Figure~\ref{fig:prism}.}
  \label{fig:steps}
\end{figure}
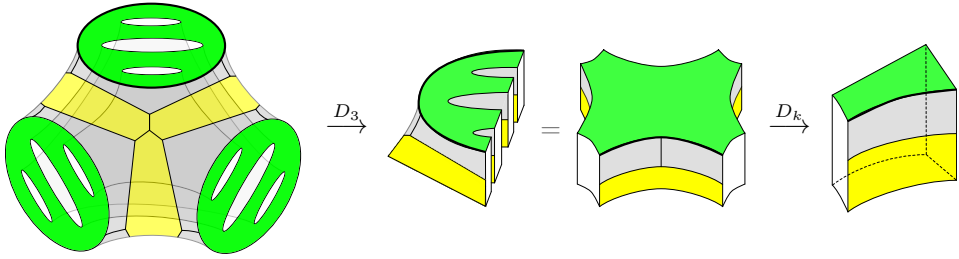

As the quotient of a manifold under the action of a reflection group, the polyhedron $B$ admits a natural  structure of a locally reflective orbifold with boundary with isotropy groups given by the stabilizers for the action of $G$.  This is not quite the structure we wish to use, though.  We augment this structure with two additional mirrors, coming from the pentagonal faces and octagonal faces, (these are the gray and yellow faces in Figure~\ref{fig:prism}, respectively).  We define isotropy groups at edges and vertices consistently with the labels in Figure~\ref{fig:prism}.  Namely, if two mirrored faces meet along an unlabeled edge, the isotropy group on that edge is dihedral of order $4$; if such an edge is labeled by a number $n$, the isotropy group is dihedral of order $2n$.  The image of $S_0$ (the green face in Figure~\ref{fig:prism}) is \emph{not} mirrored, and forms the orbifold boundary of this structure on $B$.

The underlying space of $B$ is a closed ball, but the orbifold structure suggests an abstract polyhedron, with dihedral angles of $\pi/2$ at all edges involving the unmirrored face $F_B$, and $\pi/k$ at each edge with isotropy group dihedral of order $2k$.
One can see by direct inspection that this polyhedron satisfies the hypotheses of Andreev's Theorem (see~\cite[Theorem 1.4]{RHD07}, which corrects some mistakes in the original~\cite{And70}).  
In particular, there is a unique hyperbolic polyhedron with the specified dihedral angles, and we identify this polyhedron with $B$.  Mirroring all the faces with an acute adjacent dihedral angle
 gives a hyperbolic orbifold with the same orbifold structure specified above on $B$.  In other words the polyhedron determines a hyperbolic structure on the orbifold $B$.  Since the dihedral angles adjacent to the boundary face $F_B$ are $\pi/2$, this structure has totally geodesic boundary.

Lifting the hyperbolic orbifold structure of $B$ via the quotient map $p_0: H_0\to B$ given by the action of $G$ turns $H_0$ into a hyperbolic $3$-orbifold with boundary $S_0=\partial H_0$ and $p_0$ into an orbifold covering map.  In this structure, the trivalent graph we started with divides the topological boundary of $H_0$ into pentagonal and octagonal mirrors, and the (unmirrored) subsurface $S_0$.  Faces meeting along an edge of the trivalent graph are at right angles with one another.  Thus all nontrivial isotropy groups are right-angled reflection groups of order 2, 4, or 8.
The orbifold boundary $S_0$ is totally geodesic, since it is the preimage of $F_B$ in $H_0$.
All the octagons in $Y_0$ map to a single face of $B$, denoted $F_\Sigma$ and colored in yellow in Figure~\ref{fig:prism}.

\subsection{The structure of the  boundary \texorpdfstring{$\partial H_0$}{dH0}}\label{sec:tiling_unmirrored}

The hyperbolic $3$-orbifold structure of $H_0$ described in \S\ref{sec:hyperbolization} induces a
hyperbolic $2$-orbifold structure on $S_0= \partial H_0$, which we now describe explicitly for future reference.

Each component of $\partial H_0$ is a $k$-holed sphere with a hyperbolic metric with corners, which can be subdivided into two isometric right-angled $3k$-gons. The graph $\Theta$ is dual to this subdivision; see  Figure~\ref{fig:tiling_S0}.

Each boundary circle of $\partial H_0$ is subdivided into two geodesic arcs that are mirrors for the induced orbifold structure on $\partial H_0$ (i.e., their isotropy group is $\zz / 2\zz$).
These $1$-dimensional mirrors meet orthogonally at vertices with  isotropy group  $\zz/2\zz \times \zz/2\zz$ (marked in green in Figure~\ref{fig:tiling_S0}).
The entire topological boundary of $\partial H_0$ is mirrored, so $\partial H_0$ is a closed $2$-orbifold.

We find it convenient to call an edge of the $3k$-gonal tiling of $\partial H_0$  an \textit{interior edge} if it belongs to two $3k$-gons.
The other edges of the tiling give rise to the mirrors of $\partial H_0$: two adjacent edges from different $3k$-gons form a mirror.

 \begin{figure}[ht]
  \centering
  \def\svgwidth{.4\columnwidth}
  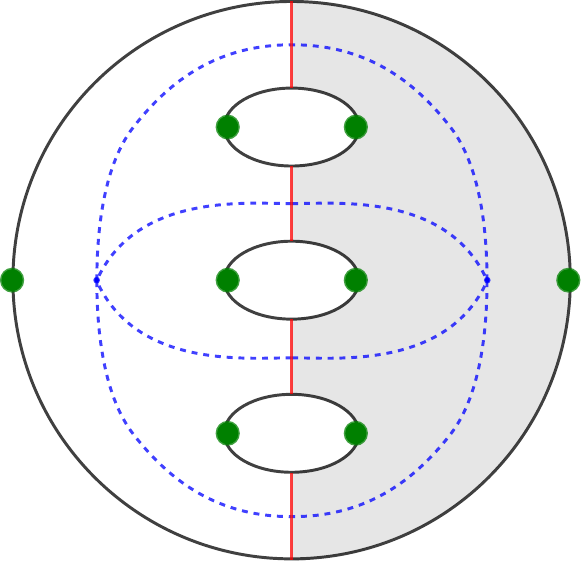
  \caption{The tiling of a component of $\partial H_0$ by two $3k$-gons for $k=4$. The dual graph $\Theta$ is represented by dotted blue arcs.
  The interior edges are marked in red.
  The green dots mark vertices with  angle $\frac \pi 2$ for the hyperbolic metric of $\partial H_0$. Unmarked vertices of the $3k$-gons have angle  $\frac \pi 2$ in each $3k$-gon, so angle $\pi$ in $\partial H_0$.}
  \label{fig:tiling_S0}
\end{figure}

\begin{choice}\label{constants}
For the rest of the paper we fix the following notation.
    \begin{enumerate}
    
    \item Let $C$ be the minimal distance between interior edges of a $3k$-gon in $\partial H_0$.    
    
    \item Let $L$ be the  length of a mirror of $\partial H_0$. Note all mirrors have the same length,  equal to twice the length of the bold edge of $B$ in Figure~\ref{fig:prism}.

    \item Let $\mu$ be the distance in $B$ from $F_B$ to the union of the two mirrored faces disjoint from $F_B$; see Figure~\ref{fig:prism}.

    \item Fix some positive $b<\mu$, and
    some $R> \frac{2\pi}{\sinh(b)}$. 
    These inequalities will allow us to apply the results from \S\ref{sec:coneoff}.
\end{enumerate}
\end{choice}
The constants $C$, $L$, and $\mu$ can be explicitly computed from the Gram matrix of $B$.  
For example, when $k=18$, we have
\[ C\approx 1.4133,\quad
L\approx 2.3619,\mbox{ and }
\mu\approx 0.069503.\]
 

\section{Finding covers with controlled geometry}\label{sec:covers}
In the last section we endowed $H_0$ with the structure of a hyperbolic $3$-orbifold with totally geodesic boundary.  In this section we construct first a topological cover $H$ of $H_0$, to which we lift the orbifold structure, and then an orbifold cover $M$ of $H$ with empty orbifold locus. (Thus $M$ is a hyperbolic \emph{manifold} with totally geodesic boundary.) See Figure~\ref{fig:coverings} for reference on the steps of the construction.
In order to guarantee that the cone-off $\coneoff M$ is negatively curved, we need to maintain control of the hyperbolic geometry of these covers at each step.

\subsection{The  orbifold \texorpdfstring{$H=H_n$}{H=Hn} covering \texorpdfstring{$H_0$}{H0}}\label{sec:topcover}
Let $p_n:H_n\to H_0$ be the topological covering map that corresponds to the covering map $T_n\to T_0$ from Definition~\ref{def:Tn}.
Note that $H_n$ is also a handlebody.
We turn $H_n$ into a hyperbolic $3$-orbifold with boundary by lifting the orbifold structure on $H_0$ that was described in \S\ref{sec:hyperbolization}.

Recall from \S\ref{sec:orbifolds} that the (orbifold) boundary $\partial H_n$ of $H_n$ is the closure of the set of points in the topological boundary of $H_n$ which have trivial isotropy group.   
The topological boundary of $H_n$ is the union of $\partial H_n$ with a surface $Y_n$ which covers the union of pairs of pants $Y_0$ and is tiled by pentagonal and octagonal mirrors.

Each component of $\partial H_n$ is the cover of some $k$-holed sphere component of $\partial H_0$,
corresponding to the covering map $\Lambda_n\to \Theta$ from \S\ref{sec:turnover groups}.
In particular, each such component is totally geodesic in $H_n$.
Recall $k\geq 18$ was fixed in Choice of constants~\ref{choice:prime}.

\begin{remark}
Each component of $\partial H_n$ is a $2$-dimensional thickening of a cover of a theta graph.  A cover of a theta-graph can be non-planar (eg $K_{3,3}$), so it is possible that these components have positive genus.
Indeed, they must: if they were planar, then the coned-off complex would be homotopy equivalent to a $3$-manifold with boundary, so its fundamental group could not contain an infinite subgroup with property (T), see \cite{FU99}.
\end{remark}

Recall from \S\ref{sec:tiling_unmirrored} that each component of $\partial H_0$
is tiled by two isometric right-angled $3k$-gons; see Figure~\ref{fig:tiling_S0}.
Lifts of these $3k$-gons   tile the components of $\partial H_n$; the dual graph to the tiling on one such component is $\Lambda_n$.
As we did for $\partial H_0$, we call \textit{interior edges} 
the edges of this tiling that belong to two $3k$-gons. Recall from Choice of constants~\ref{constants} that we denote by $C$ the minimal distance between two interior edges of a $3k$-gon. (This distance is the same for the tilings of $\partial H_0$ and $\partial H_n$.)

Moreover, the hyperbolic $2$-orbifold structure of $\partial H_0$ lifts to a hyperbolic $2$-orbifold structure on $\partial H_n$. Each mirror in $\partial H_n$ still consists of two  non-interior edges of the $3k$-gonal tiling.
\begin{definition}\label{def:systole}
    By the \emph{systole} of an orbifold we will mean the length of a shortest essential loop which misses the orbifold locus.  We write, for example, the systole of $\partial H_n$ as $\sys{\partial H_n}$.
\end{definition}

\begin{lemma}\label{lem:large systole 1}
Under the above notation, we have  
\[\sys{\partial H_n}\geq C \operatorname{girth}(\Lambda_n).\] In particular,  $\underset{n\to +\infty}{\lim} \sys{\partial H_n}=+\infty$. 
\end{lemma}
\begin{proof}
Each $3k$-gon in  $\partial H_n$ can be deformation retracted to a $k$-pod by collapsing along its interior edges. 
This provides a retraction $r:S \to \Lambda_n$ for any component $S$ of $\partial H_n$; see Figure~\ref{fig:retraction_kpod}.

Let $\gamma$ be a  closed geodesic on $\partial H_n$ avoiding the orbifold locus.
For each $3k$-gon $X$, the intersection $\gamma \cap X$  is a  collection of geodesic arcs between interior edges of $X$; see Figure~\ref{fig:retraction_kpod}.
The length of each such arc is at least the distance between two interior edges of $X$, so at least $C$. 
Let $d_n$ be the degree of the cover $\Lambda_n\to \Theta$  from \S\ref{sec:turnover groups} and let $X_1,\dots, X_{2d_n}$ be the $3k$-gons that tile the component $S$ of $\partial H_n$ containing $\gamma$.  Let $k_i$ be the number of components of $\gamma \cap X_i$.
The image of $\gamma$ under the retraction $r:S\to \Lambda_n$ is an edge-loop in $\Lambda_n$  (possibly not embedded but without backtracking) of combinatorial length $\sum_{i=1}^{2d_n} k_i$.
Thus the length of $\gamma$ satisfies 
$$\ell(\gamma) \geq  \sum_{i=1}^{2d_n} C k_i = C \sum_{i=1}^{2d_n} k_i \geq C \operatorname{girth}(\Lambda_n).$$
The last part of the statement follows from the fact that the sequence $\Lambda_n$ was chosen in Theorem~\ref{thm:tower} so that 
$\operatorname{girth}(\Lambda_n)\to \infty$.
\end{proof}

\begin{figure}[ht]
    \centering
    \def\svgwidth{.5\columnwidth}
\begingroup%
  \makeatletter%
  \providecommand\color[2][]{%
    \errmessage{(Inkscape) Color is used for the text in Inkscape, but the package 'color.sty' is not loaded}%
    \renewcommand\color[2][]{}%
  }%
  \providecommand\transparent[1]{%
    \errmessage{(Inkscape) Transparency is used (non-zero) for the text in Inkscape, but the package 'transparent.sty' is not loaded}%
    \renewcommand\transparent[1]{}%
  }%
  \providecommand\rotatebox[2]{#2}%
  \newcommand*\fsize{\dimexpr\f@size pt\relax}%
  \newcommand*\lineheight[1]{\fontsize{\fsize}{#1\fsize}\selectfont}%
  \ifx\svgwidth\undefined%
    \setlength{\unitlength}{291.11978846bp}%
    \ifx\svgscale\undefined%
      \relax%
    \else%
      \setlength{\unitlength}{\unitlength * \real{\svgscale}}%
    \fi%
  \else%
    \setlength{\unitlength}{\svgwidth}%
  \fi%
  \global\let\svgwidth\undefined%
  \global\let\svgscale\undefined%
  \makeatother%
  \begin{picture}(1,1.00693238)%
    \lineheight{1}%
    \setlength\tabcolsep{0pt}%
    \put(0,0){\includegraphics[width=\unitlength,page=1]{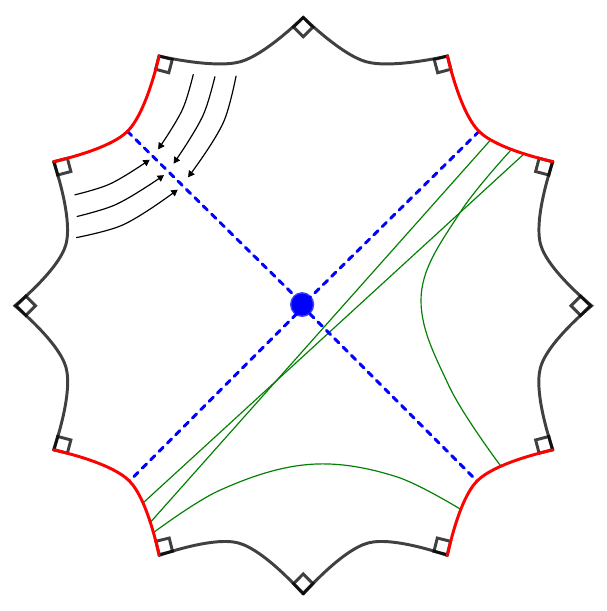}}%
    \put(0.61558975,0.26507415){\color[rgb]{0,0.50196078,0}\transparent{0.99704099}\makebox(0,0)[t]{\lineheight{1.25}\smash{\begin{tabular}[t]{c}$\gamma$\end{tabular}}}}%
    \put(0.42370482,0.74520946){\color[rgb]{0,0,0}\makebox(0,0)[t]{\lineheight{1.25}\smash{\begin{tabular}[t]{c}$r$\end{tabular}}}}%
  \end{picture}%
\endgroup%

    \caption{A $3k$-gon $X$ in the tiling of a component $S$ of $\partial H_n$, for $k=4$, with the interior edges in red, the $k$-pod to which $X$ retracts dotted in blue, the retraction $r$ sketched with black arrows, and the intersection with a geodesic $\gamma$ in green.}
    \label{fig:retraction_kpod}
\end{figure}

For the rest of the paper we make the following choice, which is possible by Lemma~\ref{lem:large systole 1}. Recall that $R$ and $L$ are defined in Choice of constants~\ref{constants}.
\begin{assumption}\label{choice}
We fix $n\geq 1$ so that Lemma~\ref{lem:turnover (T)} applies and so that the covering  $p:H=H_n\to H_0$ so that $\sys{\partial H_n}>R + L$.
For the sake of readability, we will suppress $n$ from the notation for the rest of the paper.
\end{assumption}

\begin{remark}\label{constants_for_covers}
Since $p:H\to H_0$ is a topological cover (not an orbifold cover), all the mirrors in the topological boundary of $S$ are homeomorphic lifts of the ones in $\partial H_0$, so they have the same length, which is the value $L$ from Choice of constants~\ref{constants}.
\end{remark}

\subsection{The Coxeter word of an orbifold path}\label{sec:orbifold_coxeter}
In this section we consider the orbifold with boundary $H$ that was fixed in Assumption~\ref{choice} and define a map from the orbifold fundamental group of $H$ to the Coxeter group associated with the mirrors of $H$.

Recall that the orbifold structure of $H$ is given by a finite collection of pentagonal and octagonal mirrors in the topological boundary of $H$, such that any two of them are either disjoint or meet orthogonally along an edge.
We consider the right-angled Coxeter group $W$ with generating set
$T=\{t_1,\dots, t_p\}$ equal to the set of mirrors, and in which two generators commute if and only if they intersect.

Fix a basepoint for $\piorb{H}$  in  $\partial H$ that is not contained in a mirror.
We have a natural map $f:\piorb{H} \to W$, obtained by coning off the complement of a neighborhood of the union of the mirrors of $H$ and regarding the quotient as a complex of groups.
This is a kind of ``forgetful'' map, which remembers only the sequence of mirrors which an orbifold loop passes through, and loses all information about the ordinary fundamental group $\pi_1(H)$.
We have a short exact sequence (see   \cite[Example III.$\mathcal{G}$.3.9(1)]{BH99}):
\begin{equation}\label{eq:ses}
  1 \to \llangle \pi_1(H) \rrangle \to \piorb{H}\stackrel{f}{\longrightarrow} W \to 1.
\end{equation}

Here is a more combinatorial description of the relation between $\piorb H$ and $W$ that will be useful in the following.
Recall that we have fixed the generating set $T=\{t_1,\dots, t_p\}$ for $W$.
Let $T^*$ be the set of words in the alphabet $T$, and let $q:T^*\to W$ be the natural quotient map associated with the standard RACG presentation of $W$.

\begin{definition}[Coxeter word]\label{def:coxeter word}
Let $\gamma \from I \to H$ be a smooth orbifold path in $H$ (not necessarily geodesic) such that $\gamma$ meets every mirror transversely, and misses any intersection of multiple mirrors.
The \emph{Coxeter word} $\cox\gamma$ for such a path is the word $t_{j_1}\dots t_{j_m} \in T^*$, where $\gamma$ intersects transversely  the mirrors associated to $t_{j_i}$, in the given order.  If $\cox\gamma$ has minimal length amongst words representing $q(\cox\gamma)$ in $W$, then we say that $\cox \gamma$ is \emph{reduced}.

Any smooth orbifold path $\gamma$ has arbitrarily small $C^1$ perturbations (possibly moving the endpoints) which are of the above type.  We say that $w$ is \emph{a Coxeter word for} $\gamma$ if it has minimal length among those words which appear as Coxeter words for arbitrarily small deformations of $\gamma$. 
\end{definition}

Here by ``transversely'' we mean the following: given an elevation $\widetilde \gamma$ of $\gamma$ to the orbifold universal cover of $H$  (which embeds in $\hh^3$), the lift $\widetilde \gamma$ transversely crosses each elevation of a mirror it meets.  In particular neither endpoint of $\gamma$ is on a mirror.

The map $f:\piorb{H} \to W$ is then obtained as follows: take $[\gamma]\in \piorb{H} $, represent it as a orbifold loop $\gamma$,  then set $f([\gamma])=q(\cox \gamma)$, where we have defined 
$T^*$ to be the set of words in the alphabet $T$ and $q:T^*\to W$ to be the natural quotient map.
Note that  $\cox \gamma$ is the empty word exactly when $\gamma$ admits arbitrarily small perturbations that do not cross any mirror, and in this case $[\gamma]\in \pi_1(H)$.

Arcs which are orbifold homotopic rel endpoints can have different Coxeter words, but these words all represent the same element of $W$ under the map $q$.
For example, if $\gamma$ transversely intersects two mirrors at an intersection point of the two mirrors, then there are two small perturbations in the Coxeter words of which the two mirrors appear in different order.
However, in the orbifold structure on $H$, two mirrors intersect if and only if they intersect orthogonally along an edge, so the associated generators of $W$ commute.
The two Coxeter words corresponding to the two different choices will differ by a swap of two commuting generators, so they represent the same element of $W$ and have the same word length.

\subsection{The manifold \texorpdfstring{$M$}{M} orbifold covering \texorpdfstring{$H$}{H}}\label{sec:manifold cover}
In the prequel, we have constructed a mirrored handlebody $H$ such that coning off the boundary $\partial H$ gives a $3$-dimen\-sion\-al complex $\coneoff{H}$ which deformation retracts to one of the triangle complexes $\coneoff{T}_n$ described in Lemma~\ref{lem:turnover (T)}.

We want to construct a manifold $M$ (with empty orbifold locus) that is an orbifold cover of the orbifold $H$, so that the underlying handlebody $H$ embeds into $M$ in a nice way.  
Any such orbifold cover is defined by a finite index subgroup $\Gamma\leq \piorb H$, and inherits a hyperbolic structure by lifting.
We want to find one that satisfies certain geometric conditions on the lifted hyperbolic structure, needed to apply the results from \cite{KM} about the existence of negatively curved metrics on the cone-off, see \S\ref{sec:coneoff}. 
Specifically, we must bound on the injectivity radius of the boundary $\partial M$ from below.  (We warn the reader that when taking an orbifold cover $M\to H$, the systole of $\partial M$ could be smaller than the systole of $\partial H$ as defined in Definition~\ref{def:systole}.)
To achieve our lower bound, we will exploit the residual finiteness of $W$ to carefully choose a finite-index torsion-free subgroup $K$ of $W$ and let $M$ be the orbifold cover of $H$ corresponding to $\Gamma = f^{-1}(K)$, where the Coxeter group $W$ and the map $f:\piorb{H} \to W$ were defined in \S\ref{sec:orbifold_coxeter}.

We start by proving some estimates on the length of certain orbifold geodesic paths in $\partial H$.
Recall  from Choice of constants~\ref{constants} that $L$ is the  length of a mirror in $\partial H_0$ (or equivalently in $\partial H$ by Remark~\ref{constants_for_covers}), and that in Assumption~\ref{choice}  we fixed $H$ so that $\sys{\partial H}>R + L$, where $R$ was also defined in Choice of constants~\ref{constants}.
Moreover, recall from \S\ref{sec:tiling_unmirrored} that $\partial H_0$ has a natural structure of hyperbolic $2$-orbifold, in which every boundary component is subdivided into geodesic arcs, which are mirrored and meet at right angles.
This structure was lifted to $\partial H$ in \S\ref{sec:topcover}.
The Coxeter word of a smooth orbifold path in $H$ was defined in Definition~\ref{def:coxeter word}.

The next lemmas give information about orbifold geodesics in $\partial H$.

\begin{lemma}\label{lem:unreduced word}
    Let  $S$ be a  component of $\partial H$ and let $\gamma:I\to S$ be an orbifold geodesic path.
    If the Coxeter word $\cox \gamma$ is not reduced, then $\ell (\gamma)>R$.
\end{lemma}
\begin{proof}
    Since $\cox \gamma$ is not reduced, it is not the empty word and there is a subword of the form $t\omega t$, where $\omega$ is a (possibly empty) word in the only two generators $t_i,t_j\neq t$ that commute with $t$, see \cite[\S 3.4]{DA08}.
    (Note that each topological boundary component of $S$ consists of at least $6$ mirrors, by Assumption~\ref{choice} and Lemma~\ref{lem:turnover (T)}.)
    This means that there is a subarc $\alpha$ of $\gamma$ which starts and ends on the mirror $t$ so that any sufficiently small geodesic extension $\alpha_\epsilon$ of $\alpha$ satisfies $\cox{\alpha_\epsilon} = t\omega t$.
    (Here $\alpha_\epsilon$ is the geodesic arc obtained by extending $\alpha$ geodesically at each end by $\epsilon/2$.)  
      
    Now note that the arc $\alpha$ is not entirely contained in $t$, since otherwise $\alpha_\epsilon$ would have arbitrarily small $C^1$ perturbations missing $t$, and so $t$ would not appear in its Coxeter word.
    Moreover, since $\alpha$ and the mirror $t$ are both geodesic, $\alpha$ is not homotopic relative to its endpoints to a path in $t$.
    In particular, if we introduce a shortcut $\beta$ between the endpoints of $\alpha$ along $t$, then we obtain an essential closed curve in $S$, whose length is necessarily larger than $\sys{\partial H}$.
    Note that the length of $\beta$ is at most the length of the mirror $t$, which is exactly $L$.
    Then we have 
    $$ \ell (\gamma) + L \geq \ell (\alpha) + L \geq \ell (\alpha) + \ell(\beta)\geq \sys{\partial H}.$$
    Since we choose $p:H\to H_0$ in Assumption~\ref{choice} so that $\sys{\partial H}>R+L$  we then get that
    $$\ell(\gamma) \geq \sys{\partial H} - L>R.$$  
  \end{proof}

\begin{choice}\label{choice:D}
    Let $D$ be the minimal distance between non adjacent mirrors of $\partial H$ within a fixed component of $\partial H$.  
\end{choice}

We also denote by $\wlength g$ the word length of an element $g\in W$ with respect to generating set $T=\{t_1,\dots,t_p\}$ for $W$.

\begin{lemma}\label{lem:reduced word}
    Let  $S$ be a  component of $\partial H$ and let $\gamma:I\to S$ be an orbifold geodesic path.
    If the Coxeter word $\cox \gamma$ is reduced and $\wlength {\cox \gamma}\geq 3$, then $ \ell(\gamma) \geq  \left\lfloor{\frac{\wlength{\cox \gamma}}{3} } \right\rfloor D$.
\end{lemma}
\begin{proof}
    Since $\cox \gamma$ is reduced, for any subword of $\cox \gamma$ of the form $t_it_jt_k$ we have that either $i,j,k$ are all distinct, or $t_i=t_k$ does not commute with $t_j$, see \cite[\S 3.4]{DA08}.
    In the first case, at least two of the mirrors are disjoint, so we get a contribution of $D$ to the length of $\gamma$.
    In the second case $t_j$ is not adjacent to $t_i=t_k$, hence we get contribution of $D$ in this case as well.
    Since we get a contribution of at least $D$ to the length of $\gamma$ from each subword of length $3$ of $\cox \gamma$, it follows that $ \ell(\gamma) \geq  \left\lfloor{\frac{\wlength{\cox \gamma}}{3} } \right\rfloor D$. 
\end{proof}

\begin{proposition}\label{prop:manifold}
    There exists a finite-index torsion-free normal subgroup $K\lhd W$ such that  the  orbifold cover $p_M: M\to H$ corresponding to $\Gamma = f^{-1}(K)$ satisfies the following.
    \begin{enumerate}
        \item \label{item:manifold} $M$ has empty orbifold locus (in particular $M$ is a compact hyperbolic $3$-manifold, whose totally geodesic boundary is tiled by copies of the components of $\partial H$).
        
        \item \label{item:lift} $H$ lifts to $M$ (as a manifold).   
        \item \label{item:systole} If $\widetilde \gamma:I\to \partial M$ is a closed geodesic, then  $\ell(\widetilde \gamma)> R$.

        \item \label{item:arcs} If $\widetilde \gamma:I\to \partial M$ is a geodesic arc intersecting some lift $\widetilde H$ of $H$ only in its endpoints, then  $\ell(\widetilde \gamma) \ge R$.
    \end{enumerate}
\end{proposition}
\begin{proof}
The properties \eqref{item:manifold} and \eqref{item:lift} hold for any  finite-index $K\lhd W$, for example the commutator subgroup of $W$.
Indeed, if $K$ is torsion-free, we have that $\Gamma= f^{-1}(K)$ does not contain any orbifold loop, so $M$ has empty orbifold locus, i.e., it is a manifold.
Moreover, the handlebody $H$ lifts to the cover $M$, since $\pi_1(H) \leq \ker(f) \leq \Gamma = f^{-1}(K)$.

To construct a subgroup $K$ such that the associated cover also satisfies  conditions \eqref{item:systole} and \eqref{item:arcs}, we consider the set \[A=\left\{ g \in W\ \left|\ 0<\wlength{g} \leq 3 \frac RD + 3 \right.\right\},\]
where $R$ and $D$ were introduced in Choice of constants \ref{constants} and \ref{choice:D} respectively.
Since $W$ is finitely generated, $A$ is finite.
Since $W$ is residually finite, we can find a finite-index $K\lhd W$ such that $K\cap A =\{1\}$.
Since $W$ is virtually torsion-free, we can also assume $K$ is torsion-free.
Let $p_M: M\to H$ be the orbifold cover corresponding to $f^{-1}(K)$.

Let $\widetilde \gamma:I\to \partial M$ be a closed geodesic in $\partial M$ and let $\gamma = p_M \circ \widetilde \gamma$ be its projection to $\partial H$.  
Then $\gamma$ is an orbifold geodesic with image in  some component $S\subseteq \partial H$.
We may assume either that the basepoint of $\gamma$ is not on a mirror, or that it is the midpoint of a mirror and $\gamma$ is entirely contained in that mirror.

If $\cox \gamma$ is the empty word, then $\gamma$ is a closed geodesic in $S$ that avoids all mirrors, and therefore  $\ell(\widetilde \gamma)= \ell (\gamma)\geq \sys{\partial H}>R$.
So, let us assume that $\cox \gamma$ is not the empty word.
If $\cox \gamma$ is not reduced, then   $\ell(\widetilde\gamma)=\ell (\gamma)>R$ by Lemma~\ref{lem:unreduced word}.  

So, we now assume that $\cox \gamma$ is reduced.
The word $\cox\gamma$ represents an element $g$ of $K$ by normality.  Since $\cox\gamma$ is reduced, $g$ is not the identity, so $\wlength g>3\frac{R}{D} + 3$.  Lemma~\ref{lem:reduced word} gives us 
$$\ell(\widetilde \gamma) = \ell(\gamma) \geq D \left\lfloor{\frac{\wlength{\cox \gamma}}{3}} \right\rfloor > D\left(\frac{\wlength{\cox \gamma}}{3} -1\right)  > D \left(\frac RD+1-1\right) =R.$$
In all cases, we get $\ell(\widetilde \gamma)>R$, so we proved \eqref{item:systole}.



Finally, we prove \eqref{item:arcs}.
Fix a lift $\widetilde H$ of $H$ to $M$, which exists by \eqref{item:lift}.
Let $\widetilde \gamma$ be a geodesic arc in $\partial M$ intersecting $\widetilde H$ only in its endpoints.  Let $\epsilon>0$ and define $\widetilde\gamma_\epsilon$ to be the geodesic arc $\widetilde\gamma$ extended by $\epsilon/2$ at each of its endpoints.  Let $\gamma_\epsilon = p_M(\widetilde\gamma_\epsilon)$ be the image, which lies in some component $S\subseteq \partial H$.  For $\epsilon$ sufficiently small, the Coxeter words of $\gamma_\epsilon$ do not depend on $\epsilon$.

If $\cox{\gamma_\epsilon}$ is not reduced, Lemma~\ref{lem:unreduced word} implies that $\ell(\gamma_\epsilon)>R$.  Since $\ell(\gamma_\epsilon) = \ell(\widetilde\gamma) + \epsilon$, and $\epsilon$ can be as small as we like,
we deduce $\ell(\widetilde\gamma)=\ell(\gamma)\ge R$.

Suppose then that $\cox{\gamma_\epsilon}$ is reduced.
Let $\widetilde \alpha:J\to \partial M \cap \widetilde H$ be an arc connecting the endpoints of $\widetilde \gamma$, chosen so that $\widetilde\alpha(J)$ meets the frontier of $\widetilde H\cap \partial M$ in $\partial M$ only in its endpoints and so that $\widetilde\alpha$ and $\widetilde\gamma$ fit together into a smooth orbifold loop.  We parametrize this loop so the basepoint is in $\widetilde H\cap \partial M$, and not on the frontier of $\widetilde H\cap \partial M$ in $\partial M$.  The projection $\beta = p_M\circ\widetilde\beta$ satisfies $\cox\beta = \cox{\gamma_\epsilon}$ for small $\epsilon$.
Arguing as in the reduced case above, 
$\cox{\beta}$ represents a non-trivial  element $g\in K$, so $\wlength g >3 \frac RD +3$. 
Hence by Lemma~\ref{lem:reduced word} we get $\ell(\widetilde \gamma_\epsilon) =  \ell(\gamma_\epsilon)>R$.  Letting $\epsilon$ tend to zero, we deduce $\ell(\gamma)\ge R$.
\end{proof}


\section{Proofs of the main results}\label{sec:proofs}

We are now ready to collect all the tools and prove the main results stated in the introduction.
Recall that  in Choice of constants~\ref{constants} we chose $\mu$ to be the distance in the prism $B$ from $F_B$ to the union of the two faces of $B$ disjoint from it. Then we chose $b<\mu$ and $R> 2\pi/\sinh(b)$.

\begin{proof}[Proof of Theorem~\ref{thm:main}]
Let $p_M:M\to H$ be the covering space constructed in Proposition~\ref{prop:manifold}. 
We remind the reader that $n$ was fixed in Assumption~\ref{choice}, and the orbifold $H$ as well as the manifold $M$ depend on $n$.
We have suppressed the dependence on $n$ from the notation for convenience; changing $n$ provides the sequence in the statement.

By part~\eqref{item:systole} of Proposition~\ref{prop:manifold} the injectivity radius of $\partial M$ is at least $R/2$, and  recall that $R$ was chosen so that $R/2> \pi/\sinh(b)$. 
Since the inequality is strict, there is a positive number $c>\pi/\sinh(b)$ so that 
\[ \injrad(\partial M) > \frac{R}{2} > c >\frac{\pi}{\sinh(b)}.\]

We claim that $b < \bw{M}{\partial M}$.
Indeed, let $\gamma\from I\to M$ be a geodesic realizing the buffer width of $\partial M$ in $M$.  Since $\gamma$ must be orthogonal to $\partial M$ at its endpoints, the orbifold geodesic $p_M\circ p\circ p_0\circ\gamma :I\to B$ is orthogonal to the face $F_B$ at its endpoints.  In order to return to $F_B$ it must hit some mirrors disjoint from $F_B$, since all the faces incident to $F_B$ are orthogonal to it.  In particular its total length must be at least twice the distance $\mu$ from $F_B$ to the opposite faces.
It follows that $\mu\leq \bw{M}{\partial M}$.
But $b<\mu$, so the claim is proven.

We may therefore apply
Theorem~\ref{thm:KM1} 
(\cite[Theorem A]{KM}) 
to obtain a negatively curved metric $\hat d$ on $\coneoff M$.  
This proves \eqref{item:negative}.

We let $H'$ be a lift of $H$ to $M$, whose existence is guaranteed by part~\eqref{item:lift} of Proposition~\ref{prop:manifold}.
Taking any $b'\in (b,\mu)$, we see that the set $H'$ is a tame product near $N_b(\partial M)$ in the sense of Definition~\ref{def:tame}.
Moreover, $H'$ is clearly closed and locally convex in $M$.  Proposition~\ref{prop:manifold}, part~\eqref{item:arcs}  gives 
\[ \bw{\partial M}{H'\cap\partial M} \ge \frac{R}{2}>c, \]
so the
set $H'$ satisfies all the conditions of Theorem~\ref{thm:KM2} (\cite[Theorem B]{KM}).
In particular, $\coneoff{H'}$ is isotopic to a locally convex, and hence $\pi_1$-injective, subset of $\coneoff M$.  
But $\pi_1 (\coneoff{H'}) \cong \pi_1 (\coneoff T_n)$ has property (T) by Lemma~\ref{lem:turnover (T)}.  This establishes item~\eqref{item:T} of the theorem.

Finally, to prove \eqref{item:char}, note that the systole of $\partial M$ goes to infinity with $n$.  
On the other hand, the systole of a closed hyperbolic surface is bounded by a function of the genus (see \cite{MU71}), so the genus of the boundary components must be going to infinity as well.
Moreover, the Euler characteristic of $\coneoff M_n$ is the sum of the genera of the links of the singular points.
\end{proof}

The statements in Corollary~\ref{cor:main} follow from the existence of an infinite subgroup with property (T) via standard arguments. 
We include a proof for the reader's convenience.
Note that in our context property (T) is equivalent to property (FH) by  the Delorme--Guichardet Theorem, see \cite[Theorem 2.12.4]{BHV08}.

\begin{proof}[Proof of Corollary~\ref{cor:main}]
Let $K$ be the infinite subgroup of $\pi_1(\widehat{M}_n)$ with property (T).
A group with property (T) has a global fixed point whenever it acts cubically on
a (not necessarily finite dimensional) $\CAT 0$ cube complex.
(Special cases of this fact were proved first by
Niblo--Reeves~\cite{NR97} and
Niblo--Roller~\cite{NR98}; the statement
in full generality is proved by Cornulier in~\cite{CO13}.)  
In
particular if $\pi_1(\widehat{M}_n)$ acts on a $\CAT 0$ cube complex,
then $K$ must fix a point, and so the action is not proper.  This
establishes conclusion~\eqref{item:notcubulable}.

    Similarly, since $K$ has property (FH),  \cite[Corollary 2.7.3]{BHV08} implies that any  action of $K$ by isometries on a real or complex hyperbolic space has a global fixed point, which proves \eqref{item:notHaagerup}.

    Finally, to prove \eqref{item:notRFRS} assume by contradiction that $\pi_1(\coneoff M_n)$ has a RFRS subgroup of finite index $G'$.
    Then $K'=K\cap G'$ 
    is a finite index subgroup of $K$, hence inherits  property (T); see \cite{BHV08}. Moreover, $K'$ is RFRS because it is a subgroup of $G'$, so $K'$ surjects $\zz$; see \cite{AG08}.
    But $K'$ has property (T), so its abelianization is finite (see \cite[Corollary 1.3.6]{BHV08}), which leads to a contradiction.
\end{proof}

Next we prove the claim in 
Remark~\ref{rmk:T}.
\begin{proposition}\label{prop:split}
The $3$-pseudomanifolds $\coneoff M_n$ constructed in Theorem~\ref{thm:main} contain embedded locally convex and separating surfaces.
In particular, for large enough $n$, the fundamental group $\pi_1(\coneoff M_n)$
splits.
\end{proposition}
\begin{proof}
    Consider one of the octagonal mirrors $F$ in the boundary of $H_0$, see Figure~\ref{fig:patterned_pants}.
    Since the adjacent mirrors meet $F$ orthogonally, we have that a component of the preimage of $F$ is a totally geodesic surface $\Sigma$ in $M$, disjoint from $\partial M$.
    (The surface $\Sigma$ maps into the bottom left square in $B$ labeled $F_\Sigma$, the one disjoint from the edge labeled $k$ and colored in yellow in Figure~\ref{fig:prism}.)
    Note that the distance of $\Sigma$ from $\partial M$ is larger than  $b$, so the metric of $\coneoff M$ is isometric to that of $M$ in a neighborhood of $\Sigma$ by Theorem~\ref{thm:KM1}.
    In particular, $\Sigma$ remains locally convex and locally separating in $\coneoff M$.
    It follows that $\pi_1(\coneoff M)$ splits over $\pi_1(\Sigma)$.
\end{proof}

Finally, we prove the claim in Remark~\ref{rem:agol}.
To state and prove it, we must use some notions from the theory of relatively hyperbolic groups~\cite{Gromov87,Farb98,Bow99}.  A quick introduction to most of the language as we use it can be found in~\cite[Section 2.2]{MW20}.
See~\cite{EG20} for the definition of a \emph{relatively geometric action}.
If $B$ is the mirrored polyhedron from Subsection~\ref{sec:hyperbolization}, we note that the pair $(\piorb B, \{\piorb{\partial B}\})$ is relatively hyperbolic, with Bowditch boundary homeomorphic to a $2$--sphere.  Any finite index subgroup $\Gamma$ of $\piorb B$ (for example the fundamental group of one of our manifolds $M_n$) receives an induced  peripheral structure $\mathcal P$ so that $(\Gamma,\mathcal P)$ is relatively hyperbolic and so the inclusion $\Gamma\to \piorb B$ induces an identification of Bowditch boundaries.  The peripheral groups $P\in \mathcal P$ are all commensurable to $\piorb{\partial B}$.
  We write $(\Gamma,\mathcal{P})<(\piorb B,\{\piorb{\partial B}\})$ for this relationship.
\begin{proposition}\label{prop:agol}
    Let $(\Gamma,\mathcal P)<(\piorb B,\{\piorb{\partial B}\})$ be finite index.
    Let $X$ be a $\CAT 0$ cube complex, and let $(\Gamma,\mathcal P) \acts X$ be a relatively geometric action.
    Then there exist $P\in \mathcal P$ and a hyperplane stabilizer $H$  such that $H\cap P$ is infinite.
 \end{proposition}
\begin{proof}
    By contradiction, suppose there is a relatively geometric action by $\Gamma$ on a cube complex $X$ so that hyperplane stabilizers intersect parabolic subgroups only in finite subgroups.  

    The hyperplane stabilizers suffice to separate points at infinity of the Bowditch boundary of $(\Gamma,\mathcal P)$ by~\cite[Theorem 1.4]{EGN24}.
    Their images under the inclusion of $\Gamma\to \piorb B$ therefore separate points at infinity of the Bowditch boundary of $(\piorb B,\{\piorb{\partial B}\})$, which we have noted is the same as that of $(\Gamma,\mathcal P)$.
    The result \cite[Theorem 1.3]{EKN} then produces a relatively geometric action $\piorb B\acts Y$ with hyperplane stabilizers commensurable to those of $\Gamma\acts X$.
    (Note that here we need not refine the peripheral structure, since $\piorb{\partial B}$ is one-ended.)
    Because the hyperplane stabilizers in $\piorb B\acts Y$ are commensurable to the hyperplane stabilizers in $\Gamma\acts X$, they have finite intersection with conjugates of $\piorb{\partial B}$.
    
    Let $v$ be a vertex of $Y$ with infinite stabilizer, and let $e$ be an edge incident to $v$.  Since $\Stab{e}$ is (up to index $2$) equal to the intersection of $\Stab{v}$ with the stabilizer of the hyperplane dual to $e$, we deduce $\Stab{e}$ is finite.  
    In particular the action of $\Stab{v}$ on $\lk{v,Y}$ is proper.  Since the action is relatively geometric, it is cocompact, so $\Stab{v}$ also acts cocompactly on $\lk{v,Y}$.  
    It follows that there are only finitely many conjugacy classes of elements of $\Stab{v}$ which move a vertex of $\lk {v,Y}$ a combinatorial distance $3$ or less.  Any element of $\Stab{v}$ is an element of $\piorb B$ so it acts on $\hh^3$, with some translation length.  Let $u$ be the maximum of the (finitely many) translation lengths in $\hh^3$ of elements of $\piorb B$ which stabilize some vertex $v$ and move some vertex of $\lk{ v,Y}$ a distance $3$ or less.  

    We fix $n$ large enough so that Theorem~\ref{thm:main} applies and so that moreover the systole of every boundary component of $M_n$ is strictly larger than $u$.
    We can restrict the action $\piorb B \acts Y$ to the  finite index subgroup  $\pi_1(M_n)$.
    The restricted action is still relatively geometric (with respect to the induced peripheral structure), and in particular cocompact.
    By the choice of $n$,  the following condition on the action $\pi_1(M_n)\acts Y$ holds:
    \begin{equation}\label{eq:dagger}\tag{$\dagger$}
     \parbox{0.85\linewidth}{
        For every vertex $v$ of $Y$, every element of $\Stab{v}$ acts on $\lk {v ,Y}$ by moving every vertex a combinatorial distance of at least 4.}
    \end{equation}
     
Let $K$ be the normal subgroup of $\pi_1(M_n)$ generated by the fundamental groups of the boundary components (i.e., the peripheral subgroups); we have $\pi_1(\coneoff M_n) =\rightQ{ \pi_1(M_n)}{K}$.  We have an induced action of $\pi_1(\coneoff M_n)$ on $Z = \leftQ{Y}{K}$.  This action is cocompact since the action $\pi_1(M_n)\acts Y$ was cocompact. 
Since the action of $\pi_1(M_n)$ on $Y$ is relatively geometric, vertex stabilizers have finite index in a maximal parabolic, so vertex stabilizers of $\pi_1(\coneoff M_n)$ in $Z$ are trivial (note that $\pi_1(\coneoff M_n)$ is torsion-free).  
We finally claim that $Z$ is a $\CAT 0$ cube complex.  This will imply that $\pi_1(\coneoff M_n)$ is cubulated, a contradiction.

To see that $Z$ is a $\CAT 0$ cube complex, we must check  that $Z$ is  a simply connected cube complex and satisfies Gromov's link condition (that the link of every vertex is flag).  That $Z$ is a simply connected cube complex follows from~\cite[Theorem 4.1 and Proposition 4.3]{GMimproper}.  Let $\overline{v}$ be a vertex of $Z$ which is the image of a vertex $v$ in $Y$.  Since $Y$ is $\CAT 0$, 
$\lk{v,X}$ is flag.
The stabilizer of $v$ in $\pi_1(M_n)$ acts freely on this link, and by \eqref{eq:dagger} it moves every vertex a combinatorial distance of at least $4$.  Thus the quotient $\lk{\overline{v},Z}$
 is also flag.
\end{proof}

\section{Questions}\label{sec:questions}
We showed in Corollary~\ref{cor:main} that our $3$-pseudomanifold groups are not virtually RFRS.  We ask if they are residually finite.  
(It is a well-known question whether \emph{every} word hyperbolic group is residually finite.)
Note that the groups with property (T) we use in this paper (see \S\ref{sec:turnover groups}) are finite index subgroups of some of the groups constructed in \cite{LMW19}, and even those are not known to be residually finite.

\begin{question}
    Are the groups $\pi_1(\coneoff M_n)$ constructed in this paper residually finite?  Or is there a variant of our construction which preserves non-cubulabi\-li\-ty but gives residually finite groups?
\end{question}

By \eqref{item:char} in Theorem~\ref{thm:main}, as $n$ tends to infinity, the genus of the links of the singularities in $\coneoff M_n$ also tends to infinity.  
In particular, for distinct $m$ and $n$, the spaces $\coneoff M_n$ and $\coneoff M_m$ will generally not have a common finite sheeted cover. 
This at least suggests that the groups we produce lie in infinitely many commensurability classes.

However, the genera of the vertex links may not be detected by the fundamental group.
Indeed, in any $3$-pseudomanifold, a singular vertex with link of genus $\ge 2$ can be split into multiple singular vertices with links of lower genera; conversely two singular vertices can be merged by collapsing an edge-path connecting them.
Two $3$-pseudomanifold related by these operations are  homotopy equivalent but not homeomorphic.
It is unclear to us whether these surgeries can be performed while preserving negative curvature, but we suspect it is not always possible.

\begin{question}
    Does our construction give infinitely many commensurability classes of groups?  Quasi-isometry classes?
\end{question}
For all our examples, the Gromov boundary of $\pi_1(\coneoff{M_n})$ is a Pontryagin sphere.  It is shown in~\cite{FGLS25,CDSS25} that there are infinitely many quasi-isometry classes of word hyperbolic groups with Pontryagin sphere boundary.

Our examples have singularities of some astronomically large (though in principle computable) genus.  If the surgeries we mention above cannot be done in a negative curvature preserving manner, the following question has a chance of a positive answer.
\begin{question}
   Are negatively curved $3$-pseudomanifold groups cubulable, if the singularities are restricted to have low enough positive genus?  
\end{question}

Finally, we ask whether cubulability can be ensured if the metric is sufficiently nice, say polyhedral, i.e., piecewise hyperbolic.  Note that the metric we use from \cite{KM} is not polyhedral.
\begin{question}
    Is there a variant of our construction using polyhedral metrics?
\end{question}

By Remark~\ref{rmk:T}  the groups constructed in this paper do not have property (T), and  by \cite{FU99} an infinite $3$-manifold group never has property (T).

\begin{question}
    Can a negatively curved (or more generally, aspherical) closed $3$-pseudomanifold have fundamental group with property (T)?
\end{question}


\bibliographystyle{shortalpha}  
\bibliography{biblio.bib}

\end{document}